\documentclass[11pt]{amsart} 
\usepackage{amssymb,latexsym}
\usepackage[utf8x]{inputenc}
\usepackage{amsmath,amsthm}
\usepackage{amsfonts,mathrsfs}
\usepackage{cleveref}
\usepackage{enumerate,units}
\usepackage[all]{xy}
\usepackage{graphicx,url}

\theoremstyle{plain}
\newtheorem{theorem}{Theorem}[section]
\newtheorem{corollary}[theorem]{Corollary}
\newtheorem{lemma}[theorem]{Lemma}
\newtheorem{proposition}[theorem]{Proposition}
\newtheorem{fact}[theorem]{Fact}

\newtheorem*{claim}{Claim}
\newtheorem{nclaim}{Claim}
\newtheorem*{theorem*}{Theorem}
\newtheorem{introtheorem}{Theorem}

\theoremstyle{definition}
\newtheorem{definition}[theorem]{Definition}
\newtheorem{example}[theorem]{Example}

\theoremstyle{remark}
\newtheorem*{remark}{Remark}
\newtheorem*{notation}{Notation}

\numberwithin{equation}{section}
\newcommand{\forkindep}[1][]{%
  \mathrel{
    \mathop{
      \vcenter{
        \hbox{\oalign{\noalign{\kern-.3ex}\hfil$\vert$\hfil\cr
              \noalign{\kern-.7ex}
              $\smile$\cr\noalign{\kern-.3ex}}}
      }
    }\displaylimits_{#1}
  }
}
\newenvironment{claimproof}[1][\proofname]
  {%
    \proof[#1]%
  }
  {%
    \endproof%
  }

\newcommand\s[1][]{\if#1=1 \mathfrak{s}_p \else \mathfrak{s}_{p^{#1}}\fi}

\newcommand{\Zz}{{\mathbb Z}}

\newcommand{\CL}{{\mathcal L}}
\newcommand{\CK}{{\mathcal K}}

\newcommand{\CM}{{\mathcal M}}

\newcommand{\CS}{\mathcal S}

\newcommand{\0}{\emptyset}

\renewcommand{\phi}{\varphi}

\def\nsim{\raise.17ex\hbox{$\scriptstyle\sim$}}
\def\dprk{\textrm{dp-rk}}

\makeatother

\title[Strongly dependent OAGs and fields]{Strongly dependent ordered abelian groups and henselian fields}
  
\date{\today}
	\author[Y. Halevi]{Yatir Halevi$^*$}
	\thanks{$^*$The research leading to these results has received funding from the European Research Council under the European Union’s Seventh Framework Programme (FP7/2007-2013)/ERC Grant Agreement No. 291111.}
	\address{$^*$Einstein Institute of Mathematics\\
    The Hebrew University of Jerusalem\\
    Givat Ram\\
    Jerusalem 91904\\
    Israel\\}
	\email{yatir.halevi@mail.huji.ac.il}
	\urladdr{http://ma.huji.ac.il/\textasciitilde yatirh/}

  \author[A. Hasson]{Assaf Hasson$^\dagger$}
  \thanks{$^\dagger$ Supported by ISF grant No. 181/16}
  \address{$^\dagger$Department of mathematics\\
   Ben Gurion University of the Negev\\
   Be'er Sehva\\
   Israel} \email{hassonas@math.bgu.ac.il} \urladdr{http://www.math.bgu.ac.il/\textasciitilde hasson/}

\begin{document}

\begin{abstract}
	Strongly dependent ordered abelian groups have finite dp-rank. They are precisely those groups with finite spines and $|\{p\text{ prime}:[G:pG]=\infty\}|<\infty$. We apply this to show that if $K$ is a strongly dependent field, then $(K,v)$ is strongly dependent for any henselian valuation $v$. 
\end{abstract}

\maketitle

\section{introduction}
Ordered abelian groups were classified up to elementary equivalence (and beyond) by Gurevich \cite{gurevich} and Schmitt \cite{schmitt-hab} (and references therein). One significant application was the proof in \cite{GuSh} that ordered abelian groups are dependent (i.e., do not have the independence property). This result, when combined with transfer principles (such as \cite{DelHenselian} and \cite{belair}, and most recently \cite{JaSiTransfer}), reduced -- under fairly general conditions -- the task of checking whether a (pure) henselian valued field is dependent to checking whether its residue field is. 

The finer classification of henselian dependent fields, motivated mainly by Shelah's conjecture (\cite{strongdep}) that all infinite (strongly) dependent fields are separably closed, real closed or admit a definable henselian valuation, called for a finer classification of ordered abelian groups. The immediate motivation for the investigation carried out in the present paper was the lack of worked out examples of strongly dependent ordered abelian groups (and henselian fields) that are not dp-minimal. We prove, generalising the classification of dp-minimal ordered abelian groups of \cite{dpmin-val-fields}: 
\begin{introtheorem}\label{Intro:main}
	Let $G$ be an ordered abelian group. The following are equivalent 
	\begin{enumerate}
		\item $G$ is strongly dependent;
		\item $\dprk(G)<\aleph_0$;
		\item $G$ has finite spines and $|\{p\text{ prime}:[G:pG]=\infty\}|<\infty$;
		\item $G$ is elementary equivalent to a lexicographic sum of non-zero archimedean groups $\bigoplus_{i\in I}G_i$, where  
		\begin{enumerate}
		\item for every prime $p$, $|\{i\in I: pG_i\neq G_i\}|<\infty$ and
		\item $[G_i:pG_i]=\infty$ (for some $i\in I$) for only finitely many primes $p$.
		\end{enumerate}
	\end{enumerate}
\end{introtheorem}
The spines of an ordered abelian group, in the terminology of \cite{schmitt}, are (interpretable) coloured linear orders determining the first order theory of the group. To the best of our knowledge, no systematic study of ordered abelian groups with finite spines has been carried out before. In Section \ref{S:pn}, we collect a few useful facts about ordered abelian groups. In Section \ref{S:fs} we apply Schmitt's characterization of lexicographic sums of ordered archimedian groups to characterize groups with finite spines. 

Theorem \ref{Intro:main} is proved in Section \ref{S:cal-dprk}. The proof proceeds by showing that strongly dependent ordered abelian groups have finite spines and explicitly calculating the dp-rank of the latter. This is done by first calculating the dp-rank of a certain 1-based reduct of the group, and then studying the effect of re-introducing the order into that structure. 

We have recently learned that Rafel Farr\'e \cite{farre}, Alfred Dolich and John Goodrick \cite{DoGo} have obtained, independently and using different methods, some of the results concerning ordered abelian groups obtained in this paper.

In Section \ref{fields} we apply our classification of strongly dependent ordered abelian groups to the study of strongly dependent henselian fields. Our main result is:  
\begin{introtheorem}\label{T:mainfields}
	Let $K$ be strongly dependent field and $v$ any henselian valuation on $K$. Then $(K,v)$ is strongly dependent. The value group, $vK$, is stably embedded in $(K,v)$ as a pure ordered abelian group (up to one constant), and the residue field, $Kv$, is stably embedded as a pure field. 
\end{introtheorem}
As a corollary we deduce (using results of Johnson, \cite{johnson}) that strongly dependent fields are defectless (and therefore also algebraically maximal) with respect to any henselian valuation. Our study of strongly dependent valued fields builds on ideas of Jahnke and Simon (\cite{JahNIP}, \cite{JaSiTransfer}). \\

\noindent\emph{Acknowledgements} We would like to thank Franziska Jahnke for a long discussion of an earlier draft of this paper. Her comments and ideas contributed to considerably improve the paper, especially, Section \ref{fields}. We would also like to thank Nick Ramsey, Itay Kaplan and Antongiulio Fornasiero for pointing out some mistakes in an early draft. David Pierce noted in his MathSciNet review that there were some small mathematical typos in Theorem 1, we thank him for that. Finally, We thank the anonymous referee for a meticulous reading of the paper, and his or her detailed comments and suggestions.

\section{Preliminaries and notation}\label{S:pn}
Throughout the text $G$ will denote a group, usually abelian and often ordered, $\mathfrak{C}$ will denote a sufficiently saturated model of $\mathrm{Th}(G)$. By definable we will mean definable with parameters. We will need a few results from \cite{schmitt-hab}. Since this text is not readily available, we try to keep the present work as self contained as possible, referring to more accessible sources whenever we are aware of such. In particular, for the study of ordered abelian groups we chose the language of \cite{ClHa}, rather than the language used by Schmitt. The next sub-section is dedicated to a quick overview of (parts) of the language we are using, and to the basic properties of definable sets.

\subsection{Ordered abelian groups}\label{ss:oag}
Recall that an abelian group $(G;+)$ is orderd if it is equipped with a  linear ordering $<$ such that $a<b$ implies  $a+g<b+g$ for all $a,b,g\in G$. An ordered abelian group is \emph{discrete} if it has a minimal positive element, and \emph{dense} otherwise. It is \emph{archimedean} if for all $a,b\in G$ there exists $n\in \Zz$ such that $na>b$. In particular, archimedean ordered abelian groups do not have non-trivial convex subgroups. 

%

Schmitt and Gurevich \cite{gurevich, schmitt-hab} were the first to provide quantifier elimination for ordered abelian groups. For most of our needs in the present paper a slightly different language introduced by Cluckers and Halupczok in \cite{ClHa} will be more convenient. We remind some of the notation and conventions from \cite{ClHa}: 

For any $n\in \mathbb{N}$ and $a\in G\setminus nG$ let $H_n(a)$ be the largest convex subgroup of $G$ such that $a\notin H_n(a) +nG$ (equivalently, it is the largest convex subgroup not meeting $a+nG$), and $H_n(a)=0$ if $a\in nG$. By \cite[Lemma 2.1]{ClHa} the groups $H_n(a)$ are definable (uniformly in $a$) in the language of ordered abelian groups.  We set $\mathcal{S}_n:=G/\nsim$, with $a\sim a'$ if and only if $H_n(a)=H_n(a')$, and let $\mathfrak{s}_n:G\to \mathcal{S}_n$ be the canonical map, we denote $H_n(a)$ by $G_\alpha$ for $\mathfrak{s}_n(a)=\alpha$.

Since the system of convex subgroups of an ordered abelian group are linearly ordered, $\mathcal{S}_n$ is an interpretable set linearly ordered by $\alpha\leq \alpha'$ if $G_\alpha\subseteq G_{\alpha '}$.

For any $\alpha\in \mathcal{S}_n$ and $m\in\mathbb{N}$ define \[G_\alpha^{[m]}:=\bigcap \{H+mG: G_\alpha\subsetneq H\subseteq G,  \text{  $H$ a convex subgroup}\}.\] 

Other than the sorts $\mathcal{S}_p$, Cluckers-Halupczok define two more auxiliary sorts $\mathcal{T}_p$ and $\mathcal{T}_p^+$ parametrizing more definable convex subgroups of $G$. It suffices, for our needs, to know that they are intersections and unions of convex subgroups $G_\alpha$ for $\alpha$ ranging in $\mathcal{S}_p$.

\begin{remark}
	As we will need results from \cite{schmitt-hab} we note that the groups denoted $H_n(a)$  in \cite{ClHa} (and in the present text) are denoted $F_n(a)$ by Schmitt. 
\end{remark}

We conclude this section with some basic results.

\begin{fact}\label{L:prop of H}\cite[Lemmas 2.8, 2.9, 2.10]{schmitt-hab}
\begin{enumerate}
\item $H_n(a)=H_n(a+ng)$, for any $g\in G$.
\item If $H_n(a)\subsetneq H_n(b)$ then $(a+nG)\cap H_n(b)\neq \emptyset$,
\item as a result, if $H_n(a)\subsetneq H_n(b)$ then $H_n(a+b)=H_n(b)$
\item and if $H_n(a)=H_n(b)$ then $H_n(a+b)\subseteq H_n(a)$.
\item For every prime $p$, $H_{p^m}(a)=H_{p^{m+k}}(p^ka)$.
\end{enumerate}
\end{fact}

\subsection{Examples}
Some important examples of ordered abelian groups:

\begin{example}\cite[Lemma 1.19]{schmitt-hab}\label{Zp}
Let $\chi:\{\text{primes}\}\to \mathbb{N}\cup\{\aleph_0\}$ be a function and $B=\bigcup_p \{B_p : p\text{ prime}\}$ be a linearly independent subset of $\mathbb{R}$ as a $\mathbb{Q}$-vector space such that the $B_p$ are disjoint and $|B_p|=\chi(p)$.
Let $G=\sum_p \mathbb{Z}_{(p)}\otimes\langle B_p \rangle$, where $\mathbb{Z}_{(p)}=\{n/m\in\mathbb{Q}:\gcd(m,p)=1\}$ and $\langle B_p\rangle$ is the $\mathbb{Z}$ module generated by $B_p$. Due to the linear independence of $B$,
\[G=\bigoplus\{ \mathbb{Z}_{(p)}\cdot b:p\text { prime, } b\in B_p\},\] 
and thus $[G:pG]=p^{\chi(p)}$ for every prime $p$. Letting $G$ inherit the order from $\mathbb{R}$ we get a dense archimedean group with the same property.
\end{example}

\begin{example}
Any discrete archimedean group is isomorphic (as an ordered abelian group) to $\Zz$. 
\end{example}

\begin{example}\label{e:lexi-sum}
Let $(I,<)$ be an ordered set and for each $i\in I$ let $G_i$ be an ordered abelian group. Let $\prod_{i\in I} G_i$ be the direct product of the groups, as abelian groups. For $f\in \prod_{i\in I} G_i$ we define \[supp(f)=\{i\in I:f(i)\neq 0\}.\]
The \emph{Hahn-product} of the $G_i$ is the subgroup \[H:=\{f\in \prod_{i\in I} G_i: supp(f)\text{ is a well ordered subset of } I\}\] endowed with an order defined by \[f<g \Leftrightarrow f(i)<g(i)\text{ where } i=\min supp(g-f).\]
The subgroup \[\bigoplus_{i\in I} G_i=\{f\in H: supp(f)\text{ is finite}\}\] is called the \emph{lexicographic product$\backslash$sum.}

\end{example}
\subsection{Strong dependence, burden and dp-rank}
We remind the basic model theoretic definitions with which this paper is concerned:
\begin{definition}\label{D:dprk}
Let $T$ be complete theory and $\mathfrak{C}$ a sufficiently saturated model. All elements and sequences below are taken from $\mathfrak{C}$.
\begin{enumerate}

\item $T$ has an \emph{inp-pattern of depth $\kappa$ over $A$} if
there are $(b_i^{\alpha})_{i<\omega}$, where $\alpha<\kappa$, integers $k^\alpha<\omega$ and formulas $\varphi^{\alpha}(x,y^{\alpha})$ such that each system $\{\varphi^{\alpha}(x,b_i^{\alpha}):i<\omega\}$ is $k^\alpha$-inconsistent, but for any function $\eta\in\omega^{\kappa}$ the partial type $\{\varphi^{\alpha}(x,b_{\eta(\alpha)}^{\alpha}):\alpha<\kappa\}$ is consistent.
\item The \emph{burden} (over $A$) of $T$ is the supremum over all $\kappa$ such that there is an inp-pattern of depth $\kappa$ (over $A$).
\item The \emph{dp-rank} (over $A$) of $T$ is the supremum over all $\kappa$ such that there is a $b$ and a system of $\kappa$ sequences mutually indiscernible over $A$ such that none of them is indiscernible over $Ab$.
\item For a structure $M$, define $\textrm{burden}(M)$ and $\dprk(M)$, over $A$ to be $\textrm{burden}(\mathrm{Th}(M))$ and $\dprk(\mathrm{Th}(M))$ over $A$, respectively.
\item $T$ is \emph{strongly dependent} if there are no $\aleph_0$ mutually indiscernible sequences and $b$ such that none of them are indiscernible over $b$.
\end{enumerate}
\end{definition}

\begin{remark}
\begin{enumerate}
\item In the compuation of the dp-rank of a theory $T$ the parameter set $A$ appearing in the definition does not make a difference.
\item In the definition of an inp-pattern, we may assume the $(b_i^{\alpha})_{i<\omega}$ are mutually indiscernible in which case we may require only that $\{\varphi^{\alpha}(x,b_i^{\alpha}):i<\omega\}$ be inconsistent.
\end{enumerate}
\end{remark}
The above definitions are tied together by: 
\begin{fact}\cite{adler}
If T is dependent then burden$(T)=\dprk(T)$.
\end{fact}

\begin{fact}\cite[Observation 2.1]{strongdep}\label{F:shelah-observation}
$T$ is strongly dependent if and only if for any infinite indiscernible sequence $\langle\bar{a}_t:t\in I\rangle$ (the $\bar{a}_t$ may be infinite sequences themselves) and $c$ a singleton there exists a finite convex equivalence relation $E$ on $I$ such that if $s\in I$ then $\langle\bar{a}_t:t\in(s/E)\rangle$ is indiscernible over $c$.
\end{fact}

In Section \ref{fields} Shelah's expansion, $\CM^{sh}$ of a structure $\CM$, will play an important role. We remind that $\CM^{sh}$ is obtained by expanding $\CM$ with all externally definable sets. Shelah, \cite{ShDep09}, shows that if $\CM$ is dependent $\CM^{sh}$ has quantifier elimination, and is therefore dependent. It follows immediately from the above definitions (and is well known) that if $\CM$ is dp-minimal (resp., strongly dependent) then $\CM^{sh}$ is dp-minimal (resp., strongly dependent). 

\section{Ordered abelian groups with finite spines}\label{S:fs}
We start by defining our main object of interest for the remainder of the present and the following sections: 
\begin{definition}
A pure ordered abelian group $G$ has \emph{finite spines} if $\mathcal{S}_p$ is finite for all prime $p$.
\end{definition}

\begin{remark}
If $\mathcal{S}_p$ is finite for all $p$ then $\mathcal{S}_n$ is finite for all $n$ \cite[Lemma 2.2]{ClHa}.
\end{remark}

We will see in Proposition \ref{P:strgdep_implies_A_finite} that every strongly dependent ordered abelian group has finite spines. We collect a few easy or known facts about groups with finite spines. 

\begin{lemma}\label{L:finite spines, deg}
Let $G$ be an ordered abelian group with finite spines.  For $n\in \mathbb N$ denote 
\[H_n^-(g):=\bigcup \{H_n(h): g\notin H_n(h), h\in G\}.\] 
Then  \[X=\{H^-_n(g):g\in G\}=\{H_n(g):g\in G\}=Y\]
for all $n$. 
\end{lemma}
\begin{proof}
Because $\mathcal{S}_n$ is finite and convex subgroups are linearly ordered by inclusion, $X\subseteq Y$. In the other direction, if $H_n(h)$ is maximal within the set $X$ then $H_n(h)=\bigcup\{H_n(g):h\notin H_n(g)\}=H_n^-(h)$. Otherwise let $x\in H_n(h')\setminus H_n(h)$ where $H_n(h')$ is the immediate successor of $H_n(h)$ in $Y$. It is easy to see that $H_n^-(x)=H_n(h)$.
\end{proof}

\begin{proposition}\label{P:alldefconvin_finspine}
Let $G$ be an ordered abelian group with finite spines. Then $\{G_\alpha:\alpha\in \mathcal{S}_n\text{, }n\in\mathbb{N}\}$ are all the definable convex subgroups of $G$. In particular, there are only countably many definable convex subgroups.
\end{proposition}
\begin{proof}
By \cite[Theorem 4.1]{delon-farre}\footnote{By \cite[Section 1.5]{ClHa}, what Schmitt and Delon-Farr\'e denote by $A_n(g)$ is equal to $H^-_n(g)$.}, for every definable convex subgroup of (any) ordered abelian group, there exists $n\in \mathbb{N}$ such that \[H=\bigcap_{g\notin H}H^-_n(g).\]

If $G$ has finite spines, then by Lemma \ref{L:finite spines, deg}, $H=H_n(g)$ for some $n\in\mathbb{N}$ and $g\in G$.
\end{proof}

Quantifier elimination for $G$ with finite spines is considerably simpler than in the case of arbitrary ordered abelian groups: 

\begin{proposition}\label{P:QE for finite spines}
Let $G$ be an ordered abelian group with finite spines and let $\{H_{i}\}_{i<\alpha}$ be its definable convex subgroups (including \{0\}) for some $0<\alpha\leq \omega$. Then $G$ has quantifier elimination in the the following language:
\[L=L_{oag}\cup \{(x=_{H_i} y +k_{G/H_i})_{k\in \mathbb{Z}, i< \alpha}, (x\equiv_{m,H_i} y+k_{G/H_i})_{k\in \mathbb{Z},m\in\mathbb{N}, i< \alpha}\},\] where
\begin{list}{•}{}
\item for each $k\in \mathbb{Z}$, ``$x=_H y+k_{G/H}$" is defined by $\pi (x)=\pi (y)+k_{G/H}$ for $\pi: G\to G/H$ and $k_{G/H}$ denotes $k$ times the minimal positive element of $G/H$, if it exists, and $0$ otherwise.
\item for each $k\in\mathbb{Z}$ and each $m\in \mathbb{N}$, "$x\equiv_{m,H} y+k_{G/H}$" is defined by $\pi(x)\equiv_m \pi(y) +k_{G/H}$.
\end{list}
\end{proposition}
\begin{proof}
This is a direct consequence of the main theorem of \cite{ClHa}. The auxiliary sorts $\mathcal{T}_n$ and $\mathcal{T}_n^+$ do not add any new convex subgroups because they are unions or intersection of convex subgroups coming from $\mathcal{S}_n$, and $\CS_n$ is finite.
Also the ternary relation given by $x \equiv_{m,\alpha}^{[m']}y$ if and only if $x-y\in G_{\alpha}^{[m']}+mG$ is not needed, since by \cite[Lemma 2.4]{ClHa}, and the finiteness of $\mathcal{S}_n$, $G_\alpha^{[n]}=G_{\alpha'}+nG$ for some $\alpha'\in \mathcal{S}_n$.

\end{proof}

\begin{remark}
	We do not need predicates for $\pi(x)>\pi(y) +k_{G/H}$ since, for example,
	\[\pi(x)>\pi(y)+1_{G/H} \Leftrightarrow x>y\wedge x\neq_H y\wedge x\neq_H y+1_{G/H}.\]
\end{remark}

We will need the following result, due to Schmitt:

\begin{fact}\cite[Theorem 4.13]{schmitt-hab}
An ordered abelian group $G$ is elementary equivalent to a lexicographic sum of archimedean groups if and only if for all $n,m\in\mathbb{N}$ and $0\neq x\in G$ there exists $y\in G$ such that \[H_n(x)=H^-_{n\cdot m}(y).\]
\end{fact}


The application of the above fact to groups with finite spines are summed up in the next two results: 

\begin{corollary}\label{C:elem. equiv. arhcim grps}
Every ordered abelian group with finite spines is elementary equivalent to a lexicographic sum of non-zero archimedean groups.
\end{corollary}
\begin{proof}
Let $n,m\in \mathbb{N}$ and $0\neq x\in G$. Since, by \cite[Lemma 2.2]{ClHa}, $\mathcal{S}_n\hookrightarrow \mathcal{S}_{n\cdot m}$, there exists $z\in G$ such that $H_n(x)=H_{n\cdot m}(z)$ and by Lemma \ref{L:finite spines, deg} there exists $y\in G$ such that $H_n(x)=H_{n\cdot m}(z)=H_{n\cdot m}^-(y)$.
\end{proof}

Recall the notation and definitions from Example \ref{e:lexi-sum}.

\begin{lemma}\label{L:finspine_lexi}
Let $G=\bigoplus_{i\in I} G_i$ be a lexicographic product of non-zero archimedean groups.
\begin{enumerate}
\item For $g\notin nG$, \[H_n(g)=\{h\in G : \text{for all } k\leq j\text{, }h(k)=0\},\] where $j$ is the smallest index in $supp(g)$ such that $g(j)\notin nG_j$.
\item $\mathcal{S}_p$ is finite 
if and only if $|\{i\in I: G_i \text{ not $p$-divisible}\}|<\infty$.

\end{enumerate}
\end{lemma}
\begin{proof}
\begin{enumerate}
\item Straightforward calculation.
\item Let $e_i(j)=\delta_{i,j}$. It follows from (1) above that every $i\in I$ such that $G_i$ is not $p$-divisible gives a different group $H_p(e_i)\in \mathcal S_p$.
\end{enumerate}
\end{proof}

%

\begin{example}
A group $G$ with finite spines may be strongly dependent, even dp-minimal even if it has infinitely many definable convex subgroups. E.g. \[G=\bigoplus_{p\text{ prime}} \mathbb{Z}_{(p)}\]
where $\Zz_{(p)}$ is as in Example \ref{Zp}. Indeed, since $[G:pG]<\infty$ for every prime $p$, by \cite[Proposition 5.1]{dpmin-val-fields} $G$ is dp-minimal. By an easy direct calculation $G$ has finite spines (see Proposition \ref{P:strgdep_implies_A_finite} for an abstract proof). 

By Lemma \ref{L:finspine_lexi}(1) the definable convex subgroups are all of the form \[\bigoplus_{p\leq p_0} 0\oplus \bigoplus_{p>p_0} \mathbb{Z}_{(p)},\] for prime $p_0$.
\end{example}

\section{Calculating the dp-rank}\label{S:cal-dprk}
In the present section we combine the results and observations collected in the previous sections to calculate the dp-rank of ordered abelian groups with finite spines. Let $G$ be an ordered abelian group with finite spines. We consider $G$ as a structure in the language $L$ of Proposition \ref{P:QE for finite spines}. The \emph{reduct of $G$ to the group language} is the restriction of $G$ obtained by dropping the order symbol. Namely, it is $G$ considered as a structure in the language: 
\[\CL_{\text{reduct}}=\CL_{Grp}\cup \{(x=_{H_i} y +k_{G/H_i})_{k\in \mathbb{Z}, i< \alpha}, (x\equiv_{m,H_i} y+k_{G/H_i})_{k\in \mathbb{Z},m\in\mathbb{N}, i< \alpha}\}.\]

Recall that a group $(G,+,0,\dots)$ is \emph{1-based} if every definable set (of $G^n$)  is a boolean combination of cosets of $\textrm{acl}^{\textrm eq}(\emptyset)$-definable subgroups (of $G^n$). In the following, by \emph{abelian structure} we mean an abelian group $A$ with some predicates for subgroups of powers of $A$. The key fact about abelian structures is: 
\begin{fact}\cite[Theorem 4.2.8]{wagner}.
Every abelian structure is $1$-based.
\end{fact}
This will allow us to compute the dp-rank of strongly dependent groups by, first, computing the dp-rank of their reduct to the group language (using \cite{1bsd}), and then compute the effect of re-introducing the order on the dp-rank.  Of course, quantifier elimination will play a crucial role in this computation. 

\begin{proposition}\label{P:reduct-is-1bsd}
The reduct of $G$ to the language $L_{\text{reduct}}$ is $1$-based.
\end{proposition}
\begin{proof}
Consider $G$ as an abelian group with predicates for $\{H_i\}_{i<\alpha}$, it is $1$-based. Adding constants, it is still $1$-based (see \cite[Remark 4.1.8]{pillay}). The group $G$ in the language $L_{\text{reduct}}$ is a reduct of this structure (in fact, they are bi-interpretable), hence it is also $1$-based (see \cite[Proposition 4.6.4]{pillay}).

\end{proof}

In what follows we will be using the following fact. 
\begin{fact}\cite[Proposition 3.3]{1bsd}\label{F:dp-1bsd}
Let $G$ be a $1$-based group. Then there is an inp-pattern of depth $\kappa$ over $\mathrm{acl}^{\mathrm{eq}}(\emptyset)$  if and only if there exist $\mathrm{acl}^{\mathrm{eq}}(\emptyset)$-definable subgroups $(H_\alpha)_{\alpha<\kappa}$ such that for any $i_0<\kappa$  \[\left[\bigcap_{i_0\neq \alpha<\kappa} H_\alpha:\bigcap_{\alpha<\kappa}H_\alpha\right]=\infty.\]
Furthermore, if such subgroups exist, they witness an inp-pattern of depth $\kappa$, i.e there exist an indiscernible array $(b^\alpha_i)_{\alpha<\kappa,i<\omega}$, such that $\{x\in b_i^\alpha H_\alpha\}_{\alpha<\kappa,i<\omega}$ forms an inp-pattern of depth $\kappa$.
\end{fact}

\begin{remark}
The the proof of the above actually shows:
\begin{enumerate}
\item If every definable set is a boolean combination of cosets of some family $\mathcal F$ of definable groups, then the inp-pattern may be witnessed by intersections of definable groups from $\mathcal F$ (see \cite[Remark 3.3]{1bsd}).
\item The collection of subgroups witnessing such an inp-pattern of subgroups has the property that their intersection has unbounded index in any proper subintersection.
\end{enumerate} 
\end{remark}

Thus, in order to compute the dp-rank we must first study the definable subgroups. We start by collecting some useful well-known observations: 

\begin{lemma}\label{L:intersection of groups}
Let $G$ be an ordered abelian group.
\begin{enumerate}
\item Let $A\subseteq B$ and $C\subseteq D$ be subgroups of $G$ then \[(A+D)\cap (B+C)=A+(D\cap B)+C.\]
\item Let $H$ be a convex subgroup then $nG\cap H=nH$.
\item Let $H_1\subseteq\dots\subseteq H_k$ be convex subgroups and $n_1| n_2| \dots |n_k$ be integers then
\[(n_1H_1+n_2H_k)\cap (n_1H_2\cap n_3H_k)\cap\dots\cap(n_1H_{k-1}+n_kH_k)=\]
\[n_1H_1+n_2H_2+\dots+n_kH_k.\]
\item Let $H$ be a subgroup and $n=p_1^{e_1}\cdot\ldots\cdot p_k^{e_k}$ be the prime decomposition of an integer $n$, then \[H+nG=(H+p_1^{e_1}G)\cap\dots\cap (H+p_k^{e_k}G).\]
\end{enumerate}
\end{lemma}
\begin{proof}
\begin{enumerate}

\item By an old (and easy) fact due to Dedekind, the lattice of subgroups of an abelian group is modular (i.e. if $x\leq z$ then $x\vee (y\wedge z)=(x\vee y)\wedge z$), so \[(A+D)\cap (B+C)=C+\left( (A+D)\cap B\right)=C+\left( A+(D\cap B)\right)=A+(D\cap B)+C.\]
\item Let $h\in nG\cap H$, and write $ng=h$ for $g\in G$. Replacing $h$ with $-h$ if needed, we may assume that $0<g$. Since $0<g<ng$ and $ng=h\in H$, convexity of $H$, $g\in H$.
\item By induction on $k$: The case $k=1$ is clear, so we proceed to the induction step. 
\[(n_1H_1+n_2H_k)\cap\dots\cap(n_1H_{k-2}+n_{k-1}H_k)\cap(n_1H_{k-1}+n_kH_k)=\]
\[(n_1H_1+n_2H_2+\dots+n_{k-2}H_{k-2}+n_{k-1}H_k)\cap (n_1H_{k-1}+n_kH_k).\]
Since $(n_1H_1+n_2H_2+\dots+n_{k-2}H_{k-2})\subseteq n_1H_{k-1}$, we may use $(1)$ and thus it is equal to
\[(n_1H_1+n_2H_2+\dots+n_{k-2}H_{k-2})+(n_{k-1}H_k\cap n_1H_{k-1})+n_kH_k).\]
Finally, using $(2)$, we get our result.
\item This is just the Chinese remainder theorem for $\mathbb{Z}$-modules (i.e. abelian groups) in $G/H$.
\end{enumerate}
\end{proof}
The following follows directly from the definition of $H_n(g)$:
\begin{lemma}\label{L:p-div-element-in-Sp}
Let $G$ be an ordered abelian group and $H_1\subsetneq H_2$ be convex subgroups. Then $H_2/H_1$ is not p-divisible if and only if there exists $H'\in\mathcal{S}_p$ with $H_1\subseteq H'\subsetneq H_2$.
\end{lemma}
%

The following is a special case of \cite[Lemma A.2.1]{hodges}: 
\begin{fact}\cite[Lemma A.2.1]{hodges}\label{F:groups in torsion-free}
	The theory of torsion free abelian groups proves that for every $n_\alpha,\lambda_{\alpha,j}\in\mathbb{Z}$, the formula
	\[
	\exists \bar{y} \bigwedge_{\alpha\in J} \left( n_{\alpha} x +\sum_{j=1}^{|\bar y|} \lambda_{\alpha,j} y_j =0\right)
	\]
	is equivalent to $n|x$ for some integer $n$.
\end{fact}


\begin{proposition}\label{P:definable-groups}
Let $G$ be an ordered abelian group with finite spines and $\{H_i\}_{i<\kappa}$ be all the definable convex subgroups of $G$, where $H_0=\{0\}$. Then every formula in the reduct language $L_{reduct}$ is a boolean combination of cosets of subgroups of the form 
\[H_j \text{ or } H_i+p^nG \text{, for } n\geq 0 \text{ and } H_i\in \mathcal S_p.\]
\end{proposition}
\begin{proof}
For simplicity of notation we assume that the $H_i$ are enumerated by inclusion, i.e. if $\alpha<\beta<\kappa$ then $H_\alpha<H_\beta$. 

As in the proof of Proposition \ref{P:reduct-is-1bsd}, we may expand the reduct language to $\{G,+,0,\{H_i\}_{i<\kappa}\}$ (possibly with some constants). In that language, by \cite[Theorem A.1.1]{hodges}, every formula $\varphi(x, \bar b)$ is equivalent to a boolean combination of formulas of the form
\[\tag{*}
\exists \bar{y} \bigwedge_{\alpha\in J} \left( n_{\alpha} x +t_\alpha(\bar b)+\sum_{j=1}^{|\bar y|} \lambda_{\alpha,j} y_j \in H_\alpha\right),
\]
where $n_\alpha,\lambda_{\alpha,j}$ are integers, and $t_\alpha(\bar x)$ is a term.

Note that if $g_1,g_2\models \phi(x,\bar b)$ then $g_1-g_2\models \psi(x)$ where 
\[
\psi(x):=\exists \bar{y} \bigwedge_{\alpha\in J} \left( n_{\alpha} x +\sum_{j=1}^{|\bar y|} \lambda_{\alpha,j} y_j \in H_\alpha\right).
\]
So $\phi(x,\bar b)$ defines a coset of of the subgroup defined by $\psi(x)$. Thus it will suffice to show that any definable subgroup of $G$ of the form $\psi(x)$ is the intersection of subgroups of the desired form. Since $\psi(x)$ is $\emptyset$-definable, we may apply Corollary \ref{C:elem. equiv. arhcim grps}, and assume that $G=\bigoplus_{i\in I} G_i$, where all the $G_i$ are non zero archimedean ordered abelian groups.

%

By Lemma \ref{L:finspine_lexi}(1) and Proposition \ref{P:alldefconvin_finspine} all definable convex subgroups of $G$ are of the form 
\[
H_\alpha=\bigoplus_{j\leq \alpha^-} 0\oplus\bigoplus_{j>\alpha^-} G_j
\]
for some $\alpha^-\in I$. 

Because the $H_\alpha$ are enumerated by inclusion, we get that $(a_i)_{i\in I}\models \psi(x)$  if and only if for every $\beta\in J$ and $i\leq \beta^-$ 
\[\tag{**}
a_i\models  \exists \bar{y_i} \bigwedge_{\beta \ge \alpha\in J} \left(n_\alpha x_i+\sum_{j=1}^{|\bar y_i|} \lambda_{\alpha,j} y_{i,j}=0\right).
\]
Note that, by Fact \ref{F:groups in torsion-free} for a fixed $\beta\in J$ there exists $m_\beta\in \mathbb N$ such that the formula $(**)$ is equivalent to $m_\beta\mid x$ (with $m_\beta$ independent of  $i\le \beta^-$). Assume that $J=\{\beta_1,\dots,\beta_k\}$, $\beta_1<\dots<\beta_k<\kappa$

\begin{claim} 
$\psi(G)=H_{\beta_1}+m_{\beta_1}H_{\beta_2}+\dots +m_{\beta_{k-1}}H_{\beta_k}+m_{\beta_k}G$. 
\end{claim}
\begin{claimproof}
Let $g=g_{i_1}+\dots+g_{i_m}\models \psi(x)$, where $supp(g)=\{i_1,\dots,i_m\}$ and $i_1<\dots<i_m$. Since, clearly, $\psi(G)\supseteq H_{\beta_1}$, we may assume that $i_m\leq \beta_1^-$. So for each $1\le j \le m$ there exists $1\le \ell \le k$ such that 
$g_{i_j}\le \beta_{\ell}^-$ so $g_{i_j}$ satisfies the corresponding formula (**), implying that $m_{\beta_\ell} \mid g_{i_j}$. So $g_{i_j}\in m_{\beta_\ell}H_{\ell -1}$. 

The other inclusion follows in a similar way from the characterization of those elements realizing $\psi(G)$ in (**), and the fact that if $\beta>\beta'$ then
\[
\exists \bar{y} \bigwedge_{\beta \ge \alpha\in J} \left(n_\alpha x+\sum_{j=1}^{|\bar y|} \lambda_{\alpha,j} y_{j}=0\right)
\]
defines a subgroup of 
\[
\exists \bar{y} \bigwedge_{\beta' \ge {\alpha}\in J} \left(n_{\alpha} x+\sum_{j=1}^{|\bar y|} \lambda_{{\alpha},j} y_{j}=0\right).
\]
\end{claimproof}

Since for any natural numbers $n,m$ and $H\subseteq H'$ convex subgroups,
\[nH+mH'=nH+mH+mH'=gcd(n,m)H+mH',\] we may assume that $m_{\beta_1}|m_{\beta_2}|\dots |m_{\beta_k}$ and that all the $m_{\beta_i}$ are distinct.
 By Lemma \ref{L:intersection of groups}(3) this implies that $\psi(G)$ is the intersection of subgroups of the form $n_iH_i\cap n_jH_j$. Finally, we finish by applying Lemma \ref{L:intersection of groups}(4) with the observation that for every $n\mid m$ and convex subgroups $H\subseteq H'$,
\[nH+mH'=(\{0\}+nH')\cap(H+mH')\]
and 
\[H+mH'=(H+mG)\cap(H'+\{0\}).\]

%


To show that $H_i$ can be taken in $\mathcal S_p$, consider the subgroup $H_i+p^nG$. If $H_i\notin \mathcal{S}_p$ then let $H\in\mathcal{S}_p\cup\{G\}$ be such that there is no $H'\in \mathcal{S}_p$ with $H_i\subsetneq H'\subsetneq H$. Since $\mathcal{S}_p$ is finite such a subgroup $H$ exist. By Lemma \ref{L:p-div-element-in-Sp}, $H/H_i$ must be $p$-divisible. Thus $H_i=H+p^nH$ so $H_i+p^nG=H+p^nG$.

\end{proof}

\begin{remark}
	For future reference we note that the proof of the previous proposition shows that any p.p. formula  $\phi(x,\bar b)$ as in $(*)$ defines a coset of a $\0$-definable group $A\le G$ not depending on the constant $\bar b$ (or indeed, on the terms $t_\alpha$ as in $(*)$). 
\end{remark}

We will first compute the dp-rank of $G$ in the reduct language.

\begin{notation}
Let $G$ be an ordered abelian group with finite spines and $p$ a prime. Denote by $k_p$ the maximal $n$ for which there exist definable convex subgroups $H_0\subsetneq \ldots\subsetneq H_{n-1}\subsetneq H_n=G$ such that for all $i<n$,

\[[H_{i+1}/H_i:p(H_{i+1}/H_i)]=\infty.\]
\end{notation}
\begin{lemma}\label{L:kp-witnesses}
If $G$ has finite spines and $k_p=n$ there are $H_0\subsetneq \ldots\subsetneq H_{n-1}$ witnessing it such that $H_i\in \mathcal S_p$ for all $i$. 
\end{lemma}
\begin{proof}
 Take any sequence $H_0\subsetneq H_1\subsetneq\ldots\subsetneq H_{k_p-1}$ of definable convex subgroups with \[[H_{i+1}/H_i:p(H_{i+1}/H_i)]=\infty\] for every $0\leq i\leq k_p-1$. 
 Choose the $H_i$ so that a maximal number among them is in $\mathcal S_p$. Assume towards a contradiction that there is some $H_i\notin \mathcal S_p$. By Lemma \ref{L:p-div-element-in-Sp} there exists $H\in\mathcal{S}_p$ with $H_i\subseteq H\subsetneq H_{i+1}$. Take such a subgroup $H$ which is minimal possible (such a minimal subgroup exists because $\mathcal S_p$ is finite). Then (excatly) one of  $[H_{i+1}/H:p(H_{i+1}/H)]=\infty$ or $[H/H_i:p(H/H_i)]=\infty$.  By minimality of $H$ it must be that $H_i/H$ is $p$-divisible (otherwise, apply Lemma \ref{L:p-div-element-in-Sp} with $H_i$ and $H$), so $[H_{i+1}/H:p(H_{i+1}/H)]=\infty$. Replacing $H_i$ with $H$ we get a contradiction to the choice of the sequence $H_0,\dots, H_{k_p-1}$. 
\end{proof}

\begin{remark}
The above lemma gives a simple way of computing $k_p$. Writing $\mathcal{S}_p=\{H_0,\dots,H_{n-1}\}$ and denoting
\[\mathcal{S}_p^\infty=\{H_i\in\mathcal{S}_p: [H_{i+1}/H_i:p(H_{i+1}/H_i)]=\infty\},\] with $H_n=G$, the previous lemma shows that $|\mathcal{S}_p^{\infty}|=k_p$.
\end{remark}

Since $\mathcal{S}_p^\infty\subseteq \mathcal{S}_p$ we immediately have: 
\begin{lemma}\label{L:kp is finite}
If $G$ is an ordered abelian group with finite spines and $p$ a prime, then $k_p$ is finite.
\end{lemma}
%

%

In the following proposition we study subgroups of the form $H_i+p^{e_i}G$ and of the form $H_i$. In order to avoid dividing into cases, we will allow $e_i=\infty$ with the convention that $H_i+p_i^{\infty}G=H_i$.

\begin{lemma}\label{L:inp-patter-for-grps-polyregular}
Let $G$ be an ordered abelian group with finite spines, $p$ a prime number, $H_0\subsetneq H_1\subsetneq\dots \subsetneq H_{n-1}$ definable convex subgroups and $e_0<\dots<e_{n-1}$, where $e_{n-1}$ may be $\infty$. If $e_{n-1}\neq\infty$ then for every $r<n$,
\[\left[\bigcap_{r\neq i<n}(H_i+p^{e_i}G):\bigcap_{i<n} (H_i+p^{e_i}G)\right]=\infty\Longleftrightarrow\]
 \[[H_{r+1}/H_r:p(H_{r+1}/H_r)]=\infty.\]

If $e_{n-1}=\infty$ then it is always true that
\[\left[\bigcap_{i<n-1}(H_i+p^{e_i}G):\bigcap_{i<n} (H_i+p^{e_i}G)\right]=\infty,\]
and for every $r<n-1$,
\[\left[\bigcap_{r\neq i<n}(H_i+p^{e_i}G):\bigcap_{i<n} (H_i+p^{e_i}G)\right]=\infty\Longleftrightarrow\]
 \[[H_{r+1}/H_r:p(H_{r+1}/H_r)]=\infty.\]

\end{lemma}
\begin{proof}
As the lemma is elementary and involves no parameters, we may assume that $G=\bigoplus_{i\in I} G_i$, where all the $G_i$ are non zero archimedean ordered abelian groups. For every $i<n$ there exists $i^-\in I$ such that
\[H_i=\bigoplus_{j\leq i^-} 0\oplus\bigoplus_{j>i^-} G_j.\]

By Lemma \ref{L:intersection of groups}(3), $\bigcap_{i<n} (H_i+p^{e_i}G)$ is equal to 
\[H_0+p^{e_0}H_1+p^{e_1}H_2+\dots+p^{e_{n-2}}H_{n-1}+p^{e_{n-1}}G=\]
\[p^{e_{n-1}}\cdot \left( \bigoplus_{j\leq (n-1)^-}G_j  \right)\oplus\dots \oplus p^{e_0}\cdot \left(\bigoplus_{1^-<j\leq 0^-}G_j\right)\oplus \bigoplus_{j> 0^-}G_j.\]
Likewise, $\bigcap_{r\neq i<n} (H_i+p^{e_i}G)$ is equal to
\[H_0+p^{e_0}H_1+\dots+p^{e_{r-1}}H_{r+1}+p^{e_{r+1}}H_{r+2}+\dots+p^{e_{n-2}}H_{n-1}+p^{e_{n-1}}G.\]
Thus, $\left[\bigcap_{r\neq i<n}(H_i+p^{e_i}G):\bigcap_{i<n} (H_i+p^{e_i}G)\right]=\infty$ is equivalent to 
\[\left[ p^{e_{r-1}}\left(\bigoplus_{(r+1)^-<j\leq r^-} G_j\right):p^{e_r}\left(\bigoplus_{(r+1)^-<j\leq r^-} G_j\right)\right]=\infty,\]
which is equivalent (if $e_r\neq\infty$), since $G$ is torsion-free, to 
\[[H_{r+1}/H_r:p(H_{r+1}/H_r)]=\infty.\]
The case where $e_{n-1}=\infty$ is obvious since for all $i$ we have that $[H_{i+1}:H_i]=\infty$.

\end{proof}

Before computing the dp-rank we combine the above results to obtain:
\begin{proposition}\label{P:conc-on inp-patterns}
Let $G$ be an ordered abelian group with finite spines, $\{H_i\}_{i<\kappa\leq \omega}$ a collection of definable convex subgroups with $H_i\subsetneq H_j$ if $i<j$ and $\{e_i\}_{i<\kappa}\subseteq \mathbb{N}\cup\{0,\infty\}$. Assume that for every $i_0<\kappa$
\[\left[ \bigcap_{i_0\neq i<\kappa}(H_i +p^{e_i}G):\bigcap_{i<\kappa}(H_i+p^{e_i}G)\right]=\infty.\] Then
\begin{enumerate}
\item $e_{i}\neq 0$ for every $i<\kappa$, and if $e_{i_0}=\infty$ then $i_0$ is maximal in $\kappa$.
\item $i<j<\kappa$ if and only if $e_i<e_j$.
\item $\kappa\leq k_p+1$, and is, therefore, finite.
\end{enumerate}
\end{proposition}
\begin{proof}
\begin{enumerate}
\item If $e_{i_0}=0$ then $H_{i_0}+p^{e_{i_0}}G=G$ and thus
 \[\left[ \bigcap_{i_0\neq i<\kappa}(H_i +p^{e_i}G):\bigcap_{i<\kappa}(H_i+p^{e_i}G)\right]=1.\]
If $e_{i_0}=\infty$ then $H_{i_0}+p^{e_{i_0}}G=H_i$ and note that $i_0$ must be maximal in $\kappa$. For otherwise, if $i_0<i_1$ then $H_{i_0}\subsetneq H_{i_1}+p^{e_{i_1}}G$ and thus 
 \[\left[ \bigcap_{i_1\neq i<\kappa}(H_i +p^{e_i}G):\bigcap_{i<\kappa}(H_i+p^{e_i}G)\right]=1.\]
\item Assume $i<j<\kappa$, and by the above we may assume that $e_j\neq \infty$. If $e_j\leq e_i$ then, since $H_i\subsetneq H_j$, $H_i+p^{e_i}G\subsetneq H_j+p^{e_j}G$ leading to that same contradiction as above. Since both the index set and the set of $e_i$ are ordered, this also proves the other implication. 
\item For any group, $G$, if $H,K\le G$ are subgroups then $[K:K\cap H]\le [G:H]$. So any (finite) sub-family of the $H_i$ will also satisfy the assumptions of the lemma (with the associated $e_i$). So assume towards a contradiction that $\kappa>k_p+1$ and fix a sub-family of size $k_p+2$ of the $H_i$. Applying Lemma \ref{L:inp-patter-for-grps-polyregular} to this sub-family, we see that this sub-family witnesses $k_p>k_p+1$, which is absurd. 

\end{enumerate}
\end{proof}

\begin{notation} \begin{enumerate}
		\item 	For an abelian group $G$, let \[\mathbb{P}_{\infty}(G)=\{p\text{ prime}:[G:pG]=\infty\}.\]
		\item For an ordered abelian group $G$ with finite spines and prime $p\in\mathbb{P}_\infty(G)$ let $H_{p}$ denote the maximal element in $\mathcal{S}_p^\infty$.
		If there exists a definable convex subgroup strictly containing all the $\{H_{p}\}_{p\in \mathbb{P}_\infty(G)} $ then set $c_G=1$, otherwise set $c_G=0$.
	\end{enumerate}

\end{notation}

The need for introducing the error-term $c_G$ is illustrated in the following example:
\begin{example}
 Let $A_1,A_2$ be archimedean ordered abelian groups such that $A_1$ is $p$-divisible for every prime $p\neq 2$, $[A_1:2A_1]=\infty$, and $A_2$ $p$-divisible for every prime $p\neq 2,3$ and $[A_2:2A_2]=[A_2:3A_2]=\infty$. Consider $G_1=\mathbb{Q}\oplus A_1\oplus A_2$ and $G_2=\mathbb{Z}\oplus A_1\oplus A_2$. 
 
 The convex subgroups of $G_1,G_2$ are $0$, $A_2$, $A_1\oplus A_2$ (as direct summands) and $G_i$. Among those $0$ and $A_2$ are definable in both (as $H_2(g)$ for $g\in 0\oplus H_1\oplus 0$ with $g\notin 2G$). For similar reasons $H_1\oplus H_2$ is in $\mathcal S_2(G_2)$ but not in $\mathcal S_2(G_1)$. It follows that $\mathcal S_2^\infty(G_i)=\{0,H_2\}$ (because in $G_2$ we have that $[G_2/(H_1\oplus H_2): 2(G/(H_1\oplus H_2))=2$). 
 
 Similar arguments show that $S_3=\{0\}$ in both groups. Thus, $k_2=2$ and $k_3=1$ in both groups. It follows from Proposition \ref{P:alldefconvin_finspine} that $H_1\oplus H_2$ is not definable in $G_1$. So $c_{G_2}=1$ whereas $c_{G_1}=0$. We will see in the next proposition that, despite of the fact that $k_p$ is equal in both groups for all $p$, in the reduct language,  $\dprk(G_1)=3$ whereas $\dprk(G_2)=4$.

%
%
\end{example}


\begin{proposition}\label{P:dp-rank-reduct}
Let $G$ be an ordered abelian group with finite spines, considered in the reduct language. Then $\dprk(G)$ is equal to 
\[\begin{cases}
c_G+\sum_{p\in \mathbb{P}_{\infty}(G)}k_p  & \text{if $\mathbb{P}_{\infty}(G)\neq \emptyset$} \\
1 & \text{otherwise.}
\tag{$\star$} 
\end{cases}\]
\end{proposition}
\begin{proof}
Let $\{H_\alpha\}_{\alpha<\gamma\leq \omega}$ be the definable convex subgroups of $G$. If $\dprk(G)=\kappa$, then by Fact \ref{F:dp-1bsd} we may find definable subgroups $\{N_\beta\}_{\beta<\kappa}$, and an indiscernible array $(b_s^\beta)_{s<\omega,\beta<\kappa}$ such that \[\{x\in b_s^\beta +N_\beta\}_{s<\omega,\beta<\kappa}\]
is an inp-pattern of depth $\kappa$.
Furthermore, by the remark following Fact \ref{F:dp-1bsd} and Proposition \ref{P:definable-groups}, we may assume that the $N_\beta$ appearing in such an inp-pattern are of the form $H_i+p^{e_i}G$, where possibly $e_i=\infty$ (recall the convention that $H_i+p^\infty G=H_i$). So we fix once and for all such an inp-pattern of maximal depth. Call a prime $p$ meaningful for $H$ if some $H+p^{e_i}G$ appears in our fixed inp-pattern (with $e_i<\infty$). Call $p$ meaningful if it is meaningful for some $H$.  

%

Fix a meaningful $p$ and let 
\[
\{H_\alpha+p^{e_\alpha}G\}_{\alpha<m_p}
\]
 be the family of all occurrences of $p$ in our fixed inp-pattern.   By Fact \ref{F:dp-1bsd}, for every $i_0<m_p$ 
\[\left[ \bigcap_{i_0\neq \alpha<m_p} (H_\alpha+p^{e_\alpha}G):\bigcap_{\alpha<m_p} (H_\alpha+p^{e_\alpha}G)\right]=\infty.\]
This implies that if $i\neq j$ then $H_i \neq H_j$. Otherwise, assuming without loss of generality that $e_i <e_j$, we would get 
\[\left[ \bigcap_{i\neq \alpha<m_p} (H_\alpha+p^{e_\alpha}G):\bigcap_{\alpha<m_p} (H_\alpha+p^{e_\alpha}G)\right]=1.\]
This allows us to apply Proposition \ref{P:conc-on inp-patterns}, with the implication that $m_p\leq k_p+1$ (in particular $m_p$ is finite), if $H_\alpha\subsetneq H_\beta$ then $e_\alpha <e_\beta$ and $e_\alpha\neq 0$ for every $\alpha<m_p$. Also, note that necessarily $[G:pG]=\infty$, for, otherwise, this would entail $[G:H_\alpha +p^{e_\alpha}G]<\infty$ which, as noted in the proof of Proposition \ref{P:conc-on inp-patterns}(3), is impossible. 

Summing up the above observations, we may assume the inp-pattern is witnessed by a family of subgroups
\[\{H_{\alpha_p}+p^{e_{\alpha_p}}G\}_{\alpha_p<m_p, p\in\mathbb{P}_\infty(G)}.\]
There can be only one prime $p$ for which with $e_{\alpha_p}=\infty$, for some $\alpha_p$. For otherwise, we would have $H_\alpha\subsetneq H_\beta$ both arising as subgroups in the inp-pattern but this can not be (as already mentioned above). Hence for all but (maybe) one prime $p$, $m_p\leq k_p.$
This proves that $(\star)$ is an upper bound on $\dprk(G)$ and that $\kappa\leq \omega$ with equality possible only if $\mathbb{P}_\infty(G)$ is infinite. We will now show that this bound is attained.

If $\mathbb{P}_\infty(G)$ is empty, then any sequence of pairwise distinct elements gives an inp-pattern of depth $1$ (with the formula $x=y$) so assume that $\mathbb{P}_\infty(G)\neq \emptyset$.  Let $p\in\mathbb{P}_{\infty}(G)$ and let $H_{i_{p,1}}\subsetneq \dots \subsetneq H_{i_{p,k_p}}$ be $\mathcal S_p^\infty \cup \{G\}$ (so  $H_{i_{p,k_p+1}}=G$).

\begin{claim}
The subgroups $\{H_{i_{p,1}}+pG,\dots, H_{i_{p,k_p}} +p^{k_p}G\}$ witness an inp-pattern of depth $k_p$
\end{claim}
\begin{claimproof}
This follows from the paragraph concluding the statement of Fact \ref{F:dp-1bsd} using Lemma \ref{L:inp-patter-for-grps-polyregular}. To apply this last lemma note that the groups $H_{i_{p,j}}$ were chosen specifically so that they satisfy the assumptions of the lemma. 
\end{claimproof}

As a result, for every $p\in\mathbb{P}_\infty(G)$ we have an inp-pattern of depth $k_p$. The next claim shows that we can combine these inp-patterns into one large pattern:
\begin{claim}
The subgroups $\bigcup_{p\in\mathbb{P}_{\infty}(G)}\{H_{i_{p,1}}+pG,\dots, H_{i_{p,k_p}} +p^{k_p}G\}$ witness an inp-pattern of depth $\sum_{p\in\mathbb{P}_\infty(G)}k_p$.
\end{claim}
\begin{claimproof}
Since cosets are always $2$-inconsistent, we only need to check the consistency part of the definition. For each $p\in\mathbb{P}_\infty(G)$, consider a collection $\{H_{i_{p,1}}+pG+b_{i_{p,1}},\dots, H_{i_{p,k_p}} +p^{k_p}G+b_{i_{p,k_p}}\}$ of cosets with non empty intersection. Note that if $b_p$ is any element witnessing this, then $\{H_{i_{p,1}}+pG+b_p,\dots, H_{i_{p,k_p}} +p^{k_p}G+b_p\}$ defines the exact same set. So our task is to show that the (partial) type
\[\tag{*}
\bigcup_{p\in\mathbb{P}_{\infty}(G)}\{x\in H_{i_{p,1}}+pG+b_p,\dots, x\in H_{i_{p,k_p}} +p^{k_p}G+b_p\}
\]
 is consistent. 
 Note that for every $p$,
\[ p^{k_p}G+b_p\subseteq (H_{i_{p,1}}+pG+b_p)\cap\dots\cap (H_{i_{p,k_p}} +p^{k_p}G+b_p).\]
By the Chinese remainder theorem for abelian groups, there is an element $b\in G$ such that $b\equiv_{p^{k_p}G} b_{p}$ for all $p\in\mathbb{P}_{\infty}(G)$, proving the consistency of the type $(*)$ and finishing the proof of the claim.
\end{claimproof}

The last claim finishes the proof of the proposition in case $c_G=0$. 
If $c_G=1$ let $H$ be a definable convex subgroup witnessing it.
Consider the following collection of definable subgroups from above
\[\bigcup_{p\in\mathbb{P}_\infty(G)}\{H_{i_p,j}+p^jG\}_{1\leq j\leq k_p}.\]
Enumerate these subgroups by $\{A_\alpha\}_{\alpha<\lambda}$, where $\lambda=\sum_{p\in\mathbb{P}_\infty(G)}k_p$. Let $A_{\lambda}=H$. We will show that for every $i_0<\lambda+1$
\[\left[ \bigcap_{i_0\neq \alpha<\lambda+1} A_\alpha:\bigcap_{\alpha<\lambda+1} A_\alpha\right]=\infty.\] If $i_0=\lambda$ then
\[\left[ \bigcap_{\alpha<\lambda} A_\alpha:\bigcap_{\alpha<\lambda+1} A_\alpha\right] =\infty,\] as in Lemma \ref{L:inp-patter-for-grps-polyregular}.
If $i_0<\lambda$ then we need to show that  
\[\left[ \bigcap_{i_0\neq \alpha<\lambda} A_\alpha\cap H:\bigcap_{\alpha<\lambda} A_\alpha\cap H\right]=\infty.\]
But since, $(H_{i_p,j}+p^jG)\cap H=H_{i_p,j}+p^jH$. This boils down to showing that the index is $\infty$ when we do the calculation inside $H$. By quantifier elimination $H$ is a stably embedded convex subgroup of $G$. As a result, this follows from the same calculation we have done when $c_G=0$.

\end{proof}
%

The following argument is similar to the one given by Farr\'e in \cite[Theorem 6.2]{farre}.

\begin{proposition}\label{P:strgdep_implies_A_finite}
	Let $G$ be an ordered abelian group, possibly with additional structure. If $G$ is strongly dependent then $G$ has finite spines, i.e. $\mathcal{S}_p$ is finite for all $p$, and $\mathbb{P}_{\infty}(G)$ is finite.
\end{proposition}
\begin{proof}
	If $\mathbb{P}_{\infty}(G)$ is infinite then $G$ is already not strongly dependent in the group language (see, for example, \cite{1bsd}).
	
	Since $\mathcal{S}_p$ is an interpretable linear order, in order to show that it is finite it is enough, by compactness, to show that it has no infinite ascending chain. By Lemma \ref{L:p-div-element-in-Sp}, if $H_1\subsetneq H_2$ are in $\mathcal S_p$ then $H_2/H_1$ is not $p$-divisible. Therefore, if $G$ is sufficiently saturated and $\mathcal{S}_p$ is infinite we can find for all $n$ an increasing sequence $\langle \alpha_i\in \mathcal{S}_p:i<\omega\rangle$ with \[\left[G_{\alpha_{i+1}}/G_{\alpha_i}:p(G_{\alpha_{i+1}}/G_{\alpha_i})\right]>n,\] for every $i<\omega$. By compactness and saturation we may find such an increasing sequence $\langle \beta_i\in \mathcal{S}_p:i<\omega\rangle$ such that for every $i<\omega$
\[\left[G_{\beta_{i+1}}/G_{\beta_i}:p(G_{\beta_{i+1}}/G_{\beta_i})\right]=\infty.\]
As in the proof of Proposition \ref{P:dp-rank-reduct}, by Lemma \ref{L:inp-patter-for-grps-polyregular} and Fact \ref{F:dp-1bsd}, the definable subgroups \[\{G_{\beta_i}+p^{i+1}G :i<\omega\}\] witness an inp-pattern of depth $\omega$, contradicting strong dependence.
\end{proof}
We now proceed to reintroducing the order:

\begin{lemma}\label{L:reduct-to-order}
Let $(G;+,-,0,<,\dots)$ be an ordered abelian group, possibly with some more relational symbols and constants, admitting quantifier elimination. Let $c\in G$ and $I_1=\langle a_i:i<\omega\rangle$, $I_2=\langle b_i:i<\omega\rangle$ be mutually indiscernible sequences which are also indiscernible over $c$ in the language without the order. Then at least one of $I_1,I_2$ is indiscernible over $c$ in the full language. 

\end{lemma}
\begin{proof}
Every term $t(x_1,\dots, x_n)$ is equivalent to a term of the form \[\sum_{i=1}^n z_i\cdot x_i+d,\] where $z_i\in \mathbb{Z}$ and $d$ is a  $\mathbb{Z}$-linear combination of constants. Thus every quantifier free formula in the ordered group language, not using equality, is equivalent to $\sum_{i=1}^n a_i\cdot x_i+d>0.$

Assume towards a contradiction that there are terms $t_1(\bar{x})$ and $t_2(\bar{y})$, of the above form, such that $t_1(\bar a_I)<zc< t_1(\bar a_{I'})$  and $t_2(\bar b_J)<wc<t_2(\bar b_{J'})$
where $w,z\in\mathbb{N}$ and $I,I',J, J'\subseteq \omega$ are some index sets of the appropriate lengths.
By replacing $t_1$ with $wt_1$, $t_2$ with $zt_2$ and $c$ with $wzc$ we may assume that
\[t_1(\bar a_I)<c<t_1(\bar a_{I'})\] and that
\[t_2(\bar b_J)<c<t_2(\bar b_{J'}).\]
Without loss $t_2(\bar b_J)\leq t_1(\bar a_I)<c$ so $t_1(\bar a_I)<t_2(\bar b_J')$, contradicting mutual indiscernibility.
\end{proof}

\begin{proposition}\label{P:reduct-leq-p1}
	Let $(G;+,-,0,<,\{H_i\}_{i<\omega})$ be an ordered abelian group with finite spines, possibly with some more constants, admitting quantifier elimination. Then, in the above notation, $\dprk(G)\leq \sum_{p\in \mathbb{P}_{\infty}(G)}k_p+1$. In particular, if $c_G=1$ then $\dprk_{\text{reduct}}(G)=\dprk(G)$.
\end{proposition}
\begin{proof}
	Because $G$ has finite spines $\dprk(G)\le \aleph_0$, and in case $\dprk(G)=\aleph_0$ the proposition is proved. So we assume that $G$ is of finite dp-rank. Let $\kappa=\sum_{p\in \mathbb{P}_{\infty}(G)}k_p$ and $\langle I_i:i<\kappa +2 \rangle$ be a sequence of mutually indiscernible sequences. Fix some $c\in G$. We will show that at least one of the $I_i$ is indiscernible over $c$. 
	
	If there are two sequences $I_{i_1}$ and $I_{i_2}$, both indiscernible over $c$ in the reduct language then by Lemma \ref{L:reduct-to-order} at least one of them is indiscernible over $c$ in the full language. 
	We may thus assume that 
	there is a most one of the $I_i$ which is indiscernible over $c$ in the reduct language and as $\dprk_{\text{reduct}}(G)\leq\kappa+1$,  such $I_i$ does exist. 
	
	Assume, without loss of generality, that $I_0$ is indiscernible over $c$ in the reduct language, but not indiscernible over $c$ in the full language, furthermore assume that for $i>0$, $I_{i}$ is not indiscernible over $c$ in the reduct language. Consequently, for each such $i>0$, there is a formula $\varphi_i(\bar x_i,c)$ in the reduct language witnessing this. Namely, if $I_i=\langle a_{i,j}:j<\omega\rangle$ then $\varphi_i(a_{J_{i,1}},c)$ and $\neg\varphi_i(a_{J_{i,2}},c)$ for some $J_{i,1},J_{i,2}\subseteq \{i\}\times \omega$ of the same order type. 
	
	Note that if $\varphi_i(\bar x,c)$ is a boolean combination of some formulas, then already one of the formulas in this combination witnesses non-indiscernibility over $c$. By Proposition \ref{P:definable-groups} and the remark following it we may assume that for every tuple $\bar b$ and each $\varphi_i(\bar x_i,x)$, the formula $\varphi_i(\bar b,x)$ defines a coset of  $A_i$, where $A_i$ is a definable subgroup of the form $H_i$ or $H_i+p_i^{e_i}G$ with $H_i\in \mathcal S_{p_i}$.
	
	\begin{claim}
		\begin{enumerate}
			\item For every distinct $i,j>0$, $A_i\not\subseteq A_j$.
			\item For every distinct $i,j>0$, if $A_i, A_j$ are of the form $H_i+p^{e_i}G$, $H_j+p^{e_j}G$ then $p\in \mathbb{P}_\infty (G)$, and either $[A_i:A_j\cap A_i]=\infty$ or $[A_i:A_j\cap A_i]=\infty$.
			\item For every $i>0$, $A_i$ is not of the form $H_i$.
		\end{enumerate}
	\end{claim}
	\begin{claimproof}
		Let $i,j>0$ be distinct, and assume that $A_i\subseteq A_j$. Thus $\varphi_i(a_{J_{i,1}},x)$  also defines a coset of $A_j$, and since $c$ satisfies $\varphi_i(a_{J_{i,1}},x)\wedge \varphi_j(a_{J_{j,1}},x)$, necessarily $\varphi_i(a_{J_{i,1}},x)\rightarrow \varphi_j(a_{J_{j,1}},x)$. Therefore, by mutual indiscernibility,
		\[\varphi_i(a_{J_{i,1}},x)\rightarrow \varphi_j(a_{J_{i,2}},x)\]
		but this contradicts the fact that $\neg \varphi_j(a_{J_{j,2}},c)$. 
		This gives $(1)$.
		%
		%
		
		As for $(2)$, let $I_i'$ and $I_j'$ be sequences of tuples of order type $J_{i,1}$ (resp. $J_{j,2}$) tuples in $I_i$ (resp. $I_j$) such that the convex hulls of any two tuples are disjoint. By mutual indiscernibility of $I_i$ and $I_j$, $I_i'$ and $I_j'$ are also mutual indiscernible sequences, and together with $\phi_i(\bar x_i,x)$ and $\phi_j(\bar x_j,x)$ they form an inp-pattern of depth 2. Indeed, inconsistency is clear, as for consistency, $c\models \varphi_i(a_{J_{i,1}},x)\land \varphi_j(a_{J_{j,1}},x)$ but since $I_i'$ and $I_j'$ are mutually indiscernible any path in the array is consistent. By Fact \ref{F:dp-1bsd} the desired conclusion follows. Note that this also proves that $p\in \mathbb{P}_{\infty}(G)$, for otherwise $[G:A_i]<\infty$.
%
%
		
		
		
		Finally, for $(3)$, if $A_i$ is of the form $H_i$ then, since $I_0$ is not indiscernible over $c$, we can find $J_0,J_0'\subseteq \{0\}\times\omega$ and $J_{i,1},J_{i,2}\subseteq \{i\}\times \omega$ such that, after replacing $c$ with $mc$ for some $m\in \mathbb Z$ we get
		\[t_0( a_{J_0})>c \text{ but } t_0( a_{J_0'})<c,\]
		and
		\[\varphi_i(a_{J_{i,1}},c) \text{ but } \neg \varphi_i(a_{J_{i,2}},c).\]
		We may assume that $c>0$. Note that by indiscernibility of $I_{i}$ necessarily $c\notin H_i$, indeed otherwise
		\[\varphi_i(a_{J_{i,1}},x)\leftrightarrow x\in H_i.\]
		
		Since $H_i$ is convex, necessarily
		\[\neg \varphi_i(a_{J_{i,2}},t_0(a_{J_0}))\]
		and 
		\[\varphi_i(a_{J_{i,1}},t_0(a_{J_0'})),\]
		which contradicts mutual indiscernibility.
	\end{claimproof}
	
	We can now finish the proof. If $c_G=0$ then $\dprk(G)=\kappa$ and so we must have two sequences which are indiscernible over $c$ in the reduct language, so we finish by Lemma \ref{L:reduct-to-order}.  Otherwise, $c_G=1$ and $\dprk_{reduct}(G)=\kappa+1$. By (3) of the above claim for all $i>0$,  if $A_i$ is a definable group appearing above, then $A_i$ is of the form $H_i+p_i^{e_i}G$ with $e_i<\infty$. By (2) of the claim and Lemma \ref{L:kp-witnesses}, if $ \mathcal A_p$ is the collection of all groups $A_i$ above associated with the same prime $p$, then $|\mathcal A_p|\le k_p$. By (2) again all primes $p$ appearing above belong to $\mathbb P_\infty(G)$ and by (1) of the claim they are, in particular, distinct. So, all in all, there are at most $\kappa=\sum\limits_{p\in \mathbb P_\infty} k_p$ groups $A_i$ appearing in the above. But, by assumption,  $I_i$ is associated with some definable group $A_i$ for all $0<i<k+1$. This is a contradiction.

\end{proof}

The following example shows that quantifier elimination is essential for the proposition.

\begin{example}
In the notation of Example \ref{Zp} consider $G=\bigoplus_{i<\omega} \mathbb{Z}_{(2)}$ in the language of ordered abelian group. It has infinitely many definable convex subgroups. Indeed, fixing 
\[e_i(j)=\begin{cases} 1 & \text{if } j=i\\
0 & \text{otherwise}\end{cases},\]
we get that the groups $H_2(e_i)$ (in the sense of Section \ref{ss:oag}) are all definable and distinct for $i<\omega$. But by Proposition \ref{P:strgdep_implies_A_finite}, $G$ is not strongly dependent. On the other hand, as an abelian group $G$ is dp-minimal (see, for example, \cite{1bsd}).
\end{example}

Summing up all of the above we can finally conclude our computation of the dp-rank: 

\begin{proposition}\label{P:dp-rankpoly}
Let $G$ be an ordered abelian groups with finite spines. Then $\dprk(G)=1+\sum_{p\in \mathbb{P}(G)}k_p$.
\end{proposition}
\begin{proof}
If $[G:pG]<\infty$ for every prime $p$ then $G$ is dp-minimal by \cite[Proposition 5.1]{dpmin-val-fields} in which case the proposition holds. So we may assume this is not the case.

\underline{Case 1:} Assume $c_G=0$, and hence $\dprk_{\text{reduct}}(G)=\sum_{p\in\mathbb{P}_\infty(G)}k_p$. Since $G$ is not dp-minimal, there exists a prime $q$ with $[G:qG]=\infty$. By Corollary \ref{C:elem. equiv. arhcim grps}, we may assume that $G=\bigoplus_{i\in I} G_i$, where the $G_i$ are non zero archimedean groups.
Since every discrete archimedean ordered abelian group is isomorphic to $\mathbb{Z}$ the existence of a prime $q$ such that $[G:qG]=\infty$, and the fact that $G$ is with finite spines, guarantee the existence of a dense archimedean $G_j$.

Let $(b_i)_{i<\omega}$ be an ascending indiscernible sequence of elements of the ordered set $G_j$ and $C_i$ be the definable convex subset defined by 
\[
x\in \left((\dots,0,b_i,0,\dots),(\dots,0,b_{i+1},0,\dots)\right).
\] 
The proof of Proposition \ref{P:dp-rank-reduct} provides an inp-pattern witnessing the fact that \[\dprk_{reduct}(G)=\sum_{p\in \mathbb{P}(G)}k_p.\] Our goal is to augment this inp-pattern by adjoining the formulas $\{x\in C_i\}_{i<\omega}$. By Proposition \ref{P:reduct-leq-p1}, it will suffice to show that this augmented pattern is an inp-pattern. Since inconsistency is automatic, we only have to check consistency of paths. As before, since $c_G=0$ and 
\[
p^{e_k}G\subseteq (H_{i_1}+p^{e_1}G)\cap\dots\cap (H_{i_k} +p^{e_k}G),
\] 
we only need to show that $nG\cap C_i$ is consistent for every $n\in\mathbb{N}$ and $i<\omega$. This is an easy exercise (see e.g. \cite[Lemma 1.1]{polyregular}).

\underline{Case 2:} If $c_G=1$, the result is given by Proposition \ref{P:reduct-leq-p1}.
\end{proof}

\begin{corollary}
Let $G_1$ and $G_2$ be ordered abelian groups with finite spines then
\[\dprk(G_1\oplus G_2)=\dprk(G_1)+\dprk(G_2)-1.\]
\end{corollary}

Finally as a direct corollary of Propositions \ref{P:strgdep_implies_A_finite} and \ref{P:dp-rankpoly}:
\begin{theorem}\label{T:main}
Let $G$ be an ordered abelian group. The following are equivalent 
\begin{enumerate}
\item $G$ is strongly dependent;
\item $\dprk(G)<\aleph_0$; 
\item $G$ is with finite spines and $|\mathbb{P}_{\infty}(G)|<\infty$;
\item $G$ is elementary equivalent to a lexicographic sum of non-zero archimedean groups $\bigoplus_{i\in I}G_i$, where
\begin{enumerate}
\item for every prime $p$, $|\{i\in I: pG_i\neq G_i\}|<\infty$ and
\item $[G_i:pG_i]=\infty$ (for some $i\in I$) for only finitely many primes $p$.
\end{enumerate}  
\end{enumerate}
\end{theorem}
\begin{proof}
$(1)\Rightarrow (2)$ If $G$ is strongly dependent, then by Proposition \ref{P:strgdep_implies_A_finite}, $G$ has finite spines and $\mathbb{P}_\infty(G)$ is finite and thus $\dprk(G)$ is finite by Proposition \ref{P:dp-rankpoly}.

$(2)\Rightarrow (3)$ Since every structure of finite dp-rank is strongly dependent (see Definition \ref{D:dprk}), the result follows from Proposition \ref{P:strgdep_implies_A_finite}.

$(3)\Rightarrow (4)$ Every ordered abelian group with finite spines is elementary equivalent to a lexicographic sum of non zero archimedean groups by Corollary \ref{C:elem. equiv. arhcim grps}. The rest follows from the analysis in Lemma \ref{L:finspine_lexi}.

$(4)\Rightarrow (1)$ Again, by Lemma \ref{L:finspine_lexi} it is easily seen that $G$ has finite spines and that $\mathbb{P}_\infty (G)$ is finite. Thus by Proposition \ref{P:dp-rankpoly}, $G$ has finite dp-rank and thus strongly dependent.
\end{proof}

The following is now easy: 
\begin{corollary}\label{C:induction}
	Let $G$ be an ordered abelian group, $H\le G$ a convex subgroup. If $G/H$ and $H$ are strongly dependent as pure ordered abelian groups then so is $G$. 
\end{corollary}

\begin{proof}
	We readily get that $\mathbb{P}_\infty(G)=\mathbb{P}_\infty(H)+\mathbb{P}_\infty(G/H)$. Similarly, the $p$-spine of $G$ is naturally isomorphic to the ordered union of the $p$-spine of $H$ and the $p$-spine of $G/H$. 
\end{proof}

\section{Strongly dependent henselian fields}\label{fields}
As an application of our results on strongly dependent ordered abelian groups we show that if $(K,v)$ is henselian, with $K$ strongly dependent (as a pure field) then $(K,v)$ is strongly dependent. The heart of the proof, and the main new ingredient, will be showing that the value group $vK$ is strongly dependent.   To conclude we adapt a transfer theorem (due, essentially, to Jahnke, \cite{JahNIP}, after Johnson, \cite{johnson}) to the strongly dependent setting. For a valued field $(K,v)$ we denote by $vK$ its value group, $Kv$ its residue field and $\mathcal{O}_v$ its valuation ring. All other standard valuation theoretic terminology used in this section can be found in any textbook on the subject, e.g, \cite{FVK} or \cite{EnPr}.

The following fact will be used repeatedly
\begin{fact}\cite[Proof of Claim 5.40]{strongdep}\label{F:perfect}
Every strongly dependent field is perfect.
\end{fact}

First we show that the residue field must be strongly dependent, hence perfect.

\begin{proposition}\label{P:res-field-strong}
Let $K$ be a strongly dependent field and let $v$ be a henselian valuation on $K$. Then $Kv$ is strongly dependent.
\end{proposition}
\begin{proof}
If $Kv$ is not separably closed, then $v$ is definable in $K^{sh}$ the Shelah expansion of $K$ (\cite[Theorem A]{JahNIP}), and as $K^{sh}$ is strongly dependent so is $Kv$. 

If $Kv$ is separably closed and perfect it is algebraically closed and hence strongly dependent. If it is not perfect then by an argument of Scanlon's \cite[Proposition 3.7]{JahNIP} $v$ is definable in $K$ and hence $(K,v)$ is strongly dependent, so that $Kv$ is perfect, a contradiction.
\end{proof}

Dealing with the value group is more complicated. The valuation itself may not be definable but under mild assumptions Theorem \ref{T:main} allows us to find a definable (non-trivial) coarsening of it. We need the following:

\begin{definition}\cite{coarsening}
Let $G$ be an ordered abelian group and $p$ a prime. Then, $G$ is \emph{$p$-antiregular} if no non-trivial quotient of $G$ is $p$-divisible and $G$ has no rank one quotient. 
\end{definition}
\begin{remark}
$p$-antiregularity is an elementary property of $G$, see \cite[Section 3]{coarsening}.
\end{remark}

\begin{proposition}\label{P:antireg}
Let $G$ be a non-divisible ordered abelian group with finite spines. Then there exists a prime $p$ such that $G$ is not $p$-divisible and not $p$-antiregular.
\end{proposition}
\begin{proof}
By the above remark and by Corollary \ref{C:elem. equiv. arhcim grps} we may assume that $G=\bigoplus_{i\in I}G_i$ where all the $G_i$ are non-zero archimedean groups. Let $p$ be a prime with $G$ not $p$-divisible. Since $G$ has finite spines, $\mathcal{S}_p$ is finite and hence there is a maximal element $\alpha\in \mathcal{S}_p$. Let $g\in G$ be such that $\mathfrak{s}_p(g)=\alpha$ (i.e. $G_\alpha=H_p(g)$). By Lemma \ref{L:finspine_lexi}(1) we may assume that $|supp(g)|=1$, so if $supp(g)=\{i_0\}$ then  $g(i_0)\notin pG$ and \[H_p(g)=\bigoplus_{j\leq i_0}0\oplus \bigoplus_{j>i_0} G_j.\]
Define the following convex subgroup: \[H=\bigoplus_{j< i_0}0\oplus \bigoplus_{j\geq i_0} G_j.\]

Aiming for a contradiction, assume that $G$ is $p$-antiregular. If $G=H$ then $G/H_p(g)=G_{i_0}$, which is rank one, contradiction. Otherwise, by maximality of $H_p(g)$ and Lemma \ref{L:p-div-element-in-Sp}, $G/H$ is $p$-divisible, contradiction.

\end{proof}

Recall that if $(K,v)$ is a valued field and $u$ is a coarsening of $v$ then there exists a convex subgroup $\Delta\leq vK$ such that $uK\cong vK/\Delta$. In this situation $v$ induces a valuation $\bar{v}$ on $Ku$ with valuation ring $\{xu:x\in \mathcal{O}_v\}$, where $xu$ is the residue of $x$ in the valued field $(K,u)$, and there exists an isomorphism $\bar{v}(Ku)\cong \Delta$.

For a field $K$ and a prime $p$, let $K(p)$ be the compositum of all Galois extensions of $K$ of $p$-power degree. A field $K$ is $p$-closed if $K=K(p)$. A valued field $(K,v)$ is called $p$-henselian if $v$ extends uniquely to $K(p)$. If there exists a $p$-henselian valuation with $p$-closed residue field then there exists a unique coarsest $p$-henselian valuation whose residue field is $p$-closed. It is denoted by $v_K^p$ and called the \emph{canonical $p$-henselian valuation}. For more, and the definition of the canonical $p$-henselian valuation, see, e.g., \cite{JahKoeDef}. We can now show:

\begin{proposition}\label{P:coarsening}
	Let $K$ be a strongly dependent field. Assume that $K$ admits some henselian valuation $v$ with $vK$ non-divisible. Then $K$ admits a non-trivial $\0$-definable henselian coarsening $u$ of $v$. Moreover, if $Kv$ is separably closed and $q$ is such that $vK$ is not $q$-divisible, then $u$ may be chosen so that the convex subgroup corresponding to $u$ is $q$-divisible.
\end{proposition}
\begin{proof}
	$K$ is necessarily not separably closed, otherwise, together with Fact \ref{F:perfect} we would get that $vK$ is divisible. 
	\begin{itemize}
	\item[Case 1:] If the residue field is separably closed, and hence algebraically closed by Fact \ref{F:perfect} and Proposition \ref{P:res-field-strong}, then $K$ admits a $\0$-definable non-trivial henselian valuation by \cite[Theorem 3.10]{JahKoeDef}. As the result we care about (i.e. that we actually get a coarsening) appears only in the proof of that theorem (not in its statement) we give the details: 
	
	Let $q$ be such that $vK$ is not $q$-divisible, (so $q$ is different from the characteristic of $K$). As $Kv$ is algebraically closed, by definition, the canonical $q$-henselian valuation has a $q$-closed residue field. As $K\neq K(q)$ (since $vK$ is not $q$-divisible) it is also non-trivial (see \cite[Section 2.2]{JahKoeDef}). Denote it by $v_K^q$.  It is coarser than $v$, and -- by definition -- also coarser than the canonical henselian valuation on $K$. 
	
	If $K$ contains a primitive $q^{\text{th}}$ root of unity, then $v_K^q$ is a $\0$-definable coarsening of $v$ (\cite[Theorem 2.7]{JahKoeDef}). If $K$ does not contain a primitive $q^{\text{th}}$ root of unity, we repeat the same argument with $L:=K(\zeta_q)$ to obtain a $\0$-definable $u$ coarsening the unique extension of $v$ to $L$. Since $L$ is a $\0$-definable extension, $u|K$ is a $\0$-definable coarsening of $v$. 
	
	Finally, if $v_L^q$ is the canonical $q$-henselian valuation on $L$ then by definition $Lv_L^q=Lv_L^q(q)$. Since $[L:K]<\infty$ and $v_L^q$ extends $v_K^q$ we get that $[Lv_L^q:Kv_K^q]<\infty$. Note that $Kv_K^q$ is not real closed. Indeed, since $(K,v)$ is henselian, so is $(Kv_K^q,\bar v)$. Hence if it were real closed, by \cite[Lemma 4.3.6]{EnPr} we would get that $(Kv_K^q)\bar v=Kv$ is orderable, contradiction. We conclude that $Kv_K^q=Kv_K^q(q)$.  So any valuation on the residue field has $q$-divisible valuation group. In particular $\bar{v}(Kv_K^q)$ is $q$-divisible, as required. 

	\item[Case 2:]	If $Kv$ is not separably closed, as in Proposition \ref{P:res-field-strong}, $(K,v)$ is strongly dependent and hence so is $vK$. So by Proposition \ref{P:antireg} $vK$ is not $p$-divisible and not $p$-antiregular for some $p$. Thus, by \cite[Corollary 3.7]{coarsening} $K$ admits some $\0$-definable non-trivial henselian coarsening of $v$. 
	
	\end{itemize}
	
As any coarsening of a henselian valuation is henselian, the proposition is proved. 
\end{proof}

The following observation will not be used for the proof of our main result, but may be interesting on its own right: 

\begin{corollary}
	Let $K$ be a strongly dependent field,  $(\CK,v)$ a henselian field with $v\CK$ not divisible, and $\CK$ elementarily equivalent to $K$ (as pure fields). Then $K$ is henselian (i.e. admits a non-trivial henselian valuation).
\end{corollary}
\begin{proof}
	 By the last proposition $\mathcal K$ admits a $\0$-definable non-trivial henselian valuation. Since $K\equiv \mathcal K$ the same is true of $K$. 
\end{proof}

\begin{remark}
	Recall (\cite{PrZi1978}) that a field is $t$-henselian if it is elementarily equivalent (in the language of rings) to a henselian field. The assumptions of the last corollary are equivalent to $K$ being $t$-henselian, admitting some valuation $v$ with $vK$ non-divisible. 
\end{remark}

Using the above results we can finally conclude the following:

\begin{proposition}\label{P:value-grp-strong}
	Let $K$ be a strongly dependent field, $v$ a henselian valuation on $K$. Then the value group $vK$ is strongly dependent as a pure group. 
\end{proposition}
\begin{proof}
	If $K$ is separably closed, and hence algebraically closed, the result follows from the strong dependence of ACVF. So we assume this not to be the case. 
	
	If $\mathbb{P}_{\infty}(vK)=\emptyset$ we get by \cite[Proposition 5.1]{dpmin-val-fields} that $vK$ is dp-minimal, and we are done. So we may assume that $|\mathbb{P}_{\infty}(vK)|>0$ and fix some prime $p\in \mathbb{P}_\infty(vK)$. We may assume that $Kv$ is algebraically closed, otherwise, $v$ is $K^{sh}$-definable by \cite[Theorem A]{JahNIP}), and we are done (as in the proof of Proposition \ref{P:res-field-strong}). 
	
	Proposition \ref{P:coarsening} supplies us with a non-trivial $\0$-definable henselian coarsening $u$ of $v$. Consider $Ku$, equipped with the valuation $\bar v$. By Proposition \ref{P:coarsening}, $\bar v (Ku)$, the corresponding convex subgroup of $vK$, may be chosen to be $p$-divisible. So $\mathbb{P}_\infty(\bar v(Ku))\subsetneq \mathbb{P}_\infty(vK)$. 

\begin{claim}
$|\mathbb{P}_{\infty}(vK)|<\infty$.
\end{claim}	
\begin{claimproof}
Either by \cite[Corollary 3.12]{KaSh} or by \cite{1bsd}, since $K^\times$ is a strongly dependent abelian group $|\mathbb{P}_{\infty}(K^\times)|<\infty$. Now notice that \[|\mathbb{P}_{\infty}(vK)|\leq |\mathbb{P}_{\infty}(K^\times)|.\]
\end{claimproof}

We conclude by induction on $|\mathbb{P}_{\infty}(vK)|$: by the induction hypothesis $\bar v (Ku)$ is strongly dependent (because $\bar v$ is henselian). It follows from Corollary \ref{C:induction} that $vK$ is strongly dependent since $vK/\bar v (Ku)$ and $\bar v (Ku)$ are strongly dependent. 
\end{proof}

Recall the following definition:
\begin{definition}

A valued field $(K,v)$ of residue characteristic $p>0$ is a \emph{Kaplansky field} if the value group is $p$-divisible, the residue field is perfect and the residue field does not admit any finite separable extensions of degree divisible by $p$.

\end{definition}

In \cite[Theorem 3.3]{JaSiTransfer}, Jahnke-Simon show that any theory of separably algebraically maximal Kaplansky fields of a fixed finite degree of imperfection is dependent if and only if the residue field and value group are.

\begin{proposition}\label{P:transf-strong-dep}
Let $(K,v)$ be either an algebraically maximal Kaplansky valued field or a henselian valued field of residue characteristic $0$. Then $(K,v)$ is strongly dependent if and only if the residue field and the value group are. The result remains valid under strongly dependent expansions of the residue field.
\end{proposition}
\begin{proof}
Let $(K,v)$ be either an algebraically maximal Kaplansky valued field or a henselian valued field of residue characteristic $0$. Since strong dependence is preserved under interpretations, it is clear that if $(K,v)$ is strongly dependent so are the value group and the residue field with all their induced structure. So we prove the reverse implication.

By \cite{DelHenselian} (and using \cite{GuSh}) for the residue characteristic $0$ case and \cite{belair} for the algebraically maximal Kaplansky case, $(K,v)$ is  dependent. By \cite[Theorem 3.12]{touchard}, it is, if fact, strongly dependent.

For the last part, recall that the residue field of a strongly dependent pure valued field $(K,v)$ is a stably embedded pure field provided that $(K,v)$ is either algebraically maximal Kaplansky or henselian of residue characteristic $0$, see for example  \cite[Corollary 4.4]{qe-kap}.

Consider now a strongly dependent expansion of the residue field.  By \cite[Proposition 2.5]{JaSiTransfer}, the expanded structure on $(K,v)$ is still dependent; by  \cite[Remark 3.15]{touchard}, it is still strong, and hence strongly dependent.
\end{proof}

\begin{lemma}\label{L:Kaplansky}
	Let $K$ be a strongly dependent field of characteristic $p>0$. Then $(K,v)$ is an algebraically maximal Kaplansky field with respect to any henselian valuation $v$. Furthermore, $(K,v)$ is strongly dependent.
\end{lemma}
\begin{proof}
Since $\mathrm{char}(K)=p$ it is perfect and so $vK$ is $p$-divisible. 
Moreover, as $K$ is dependent it follows from the proof of \cite[Proposition 5.3]{KaScWa} that $Kv$ is Artin-Schreier closed, and therefore infinite. 

Recall that by \cite[Corollary 4.4]{KaScWa} infinite dependent fields of characteristic $p$ have no separable extensions of degree divisible by $p$, the characteristic of the field. Thus, strongly dependent fields, which are perfect, have no finite extensions of degree dividing $p$. The residue field $Kv$ is strongly dependent by Proposition \ref{P:res-field-strong} and hence $(K,v)$ is Kaplansky.

Since the degree of every finite extension of $K$ is prime to $p$, $K$ is defectless and thus, by henselianity, algebraically maximal. By propositions \ref{P:res-field-strong}, \ref{P:value-grp-strong} and \ref{P:transf-strong-dep} $(K,v)$ is strongly dependent.
\end{proof}

\begin{proposition}\cite[The proof of Theorem 4.3.1]{johnson}\label{P:johnson}
Let $K$ be a strongly dependent field and $(K,v)$ henselian of mixed characteristic $(0,p)$. Then 

\begin{enumerate}
\item either $[0,v(p)]$ is finite or there exists an non-trivial $p$-divisible convex subgroup of $vK$,
\item if $[0,v(p)]$ is infinite then $Kv$ is infinite.
\end{enumerate}
\end{proposition}
\begin{proof}
\begin{enumerate}
\item Assume  $[0,v(p)]$ is infinite. Let $\Delta_p$ be the maximal $p$-divisible convex subgroup of $vK$.
\begin{claim}
There is a formula defining, in any ordered abelian group, the maximal $p$-divisible convex subgroup.
\end{claim}
\begin{claimproof}
Consider \[X=\{g\in vK: \text{for all } 0\leq |x|\leq |g| \text{, $x$ is $p$-divisible}\}.\]
All elements of $X$ are obviously $p$-divisible and it is closed under inverses. Let $g,h\in X$ and assume for simplicity that $0<g+h$. We may assume that $g,h>0$ and let $0<c<g+h$. If $c\leq g$ or $c\leq h$ then $c$ is $p$-divisible so assume without loss of generality that $h<c$, but then $0<c-h<g$ hence $c-h$ is $p$-divisible and thus so is $c$. So $X$ is a subgroup. By definition we must have that $X=\Delta_p$. 
%
\end{claimproof}

As a result, what we want to prove is first order expressible so we may assume that $(K,v)$ is sufficiently saturated and specifically that $|[0,v(p)]|>|\mathbb{R}|$.

Let $\Delta$ be the minimal convex subgroup of $vK$ containing $v(p)$ and $\Delta_0$ the maximal convex subgroup not containing $v(p)$. Since $\Delta/\Delta_0$ is archimedean it embeds into $\mathbb{R}$. If $\Delta_0$ were trivial then, since $[0,v(p)]\subseteq \Delta$ necessarily $|\Delta|>|\mathbb{R}|$, which is impossible. 

The following claim will finish $(1)$.
\begin{claim}
$\Delta_0$ is $p$-divisible and thus $\Delta_0\subseteq \Delta_p$.
\end{claim}
\begin{claimproof}
The coarsening $v_0:K\to vK/\Delta$ of $v$ is henselian of equi-characteristic $0$. 
In particular $K_1:=Kv_0$ is strongly dependent by Proposition \ref{P:res-field-strong}. Also, the valuation $v_1:K_1\to \Delta/\Delta_0$ of mixed characteristic $(0,p)$ is henselian.

Finally, consider the valuation $v_2:K_2\to \Delta_0$, where $K_2:=K_1v_1$. Note that $K_2$ is of characteristic $p>0$ and that $K_2$ is strongly dependent by Proposition \ref{P:res-field-strong}. By Fact \ref{F:perfect}, $K_2$ is perfect and hence $\Delta_0$ is $p$-divisible.
\end{claimproof}
%
%
%
%
%
\item  Keeping the same notation, assume that $[0,v(p)]$ is infinite. As before, $\Delta_0$ is non trivial. The proof of \cite[Proposition 5.3]{KaScWa} shows  that if $K$ is dependent and $(K,v)$ is a valued field of characteristic $p>0$ then $Kv$ is infinite. Applying this fact to the valuation $v_2:K_2\to \Delta_0$, whose residue field is $Kv$, finishes the proof.
\end{enumerate}
\end{proof}

We can now prove the main part of Theorem \ref{T:mainfields}: 

\begin{theorem}\label{T:fields}
	Let $K$ be a strongly dependent field. Assume that $v$ is a henselian valuation on $K$ then $(K,v)$ is strongly dependent.
\end{theorem}
\begin{proof}
We may move to a sufficiently saturated extension of $(K,v)$, keeping the base field strongly dependent.
By propositions \ref{P:res-field-strong} and \ref{P:value-grp-strong}, $vK$ and $Kv$ are strongly dependent. The proof now splits in three case: 

\underline{Case 1:} If $\mathrm{char}(Kv)=0$ then, since $vK$ and $Kv$ are strongly dependent, by \cite[Claim 1.17]{strongdep}, $(K,v)$ is also strongly dependent. Moreover, we note that in this case $vK$ and $Kv$ are stably embedded as pure structures by \cite[Corollary 5.25]{vd}.

\underline{Case 2:} Assume that $\mathrm{char}(K)=\mathrm{char}(Kv)=p>0$. This case follows by Lemma \ref{L:Kaplansky}. Moreover, we note that $vK$ and $Kv$ are stably embedded as pure structures by \cite[Lemma 3.1]{JaSiTransfer}.

\underline{Case 3:} Assume that $(K,v)$ is of mixed characteristic $(0,p)$.
Let $\Delta_0$ be the largest convex subgroup of $\Gamma:=vK$ not containing $v(p)$ and $\Delta$ the smallest convex subgroup containing $v(p)$. 

The coarsening $v_0:K\to \Gamma/\Delta$ of $v$ is henselian of equi-characteristic $0$. So by Case 1, $(K,v_0)$ is strongly dependent. 
In particular $K_1:=Kv_0$ is strongly dependent. Also, the valuation $v_1:K_1\to \Delta/\Delta_0$ of mixed characteristic $(0,p)$ is henselian.

Finally, consider the valuation $v_2:K_2\to \Delta_0$, where $K_2:=K_1v_1$. It is of equi-characteristic $(p,p)$ and thus $(K_2,v_2)$ is strongly dependent by Case 2.

\underline{Case 3.1:} If $K_1v_1=K_2$ is finite then so is $Kv$, and hence it is not separably closed, by \cite[Theorem A]{JahNIP}, $v$ is definable in $K^{sh}$ so $(K,v)$ is strongly dependent. 

\underline{Case 3.2:} Assume $K_2$ is infinite.

\begin{nclaim}
$(K_1,v_1)$ is unboundedly ramified, i.e. $[0,v_1(p)]$ is infinite, and $\Delta/\Delta_0$ is $p$-divisible.
\end{nclaim}
\begin{claimproof}
If $[0,v_1(p)]$ is finite then the valuation $v_1:K_1\to \Delta/\Delta_0$ is discrete, hence by \cite[Theorem 4]{Hong} $v_1$ is definable in $K_1$, so $(K_1,v_1)$ is strongly dependent. Now, by \cite[Lemma 4.2.1]{johnson}, $K_1v_1=K_2$ is finite, a contradiction.

We may now apply Proposition \ref{P:johnson} to $(K_1,v_1)$. Since $\Delta/\Delta_0$ is archimedean, if it contains a non-trivial $p$-divisible convex subgroup, $\Delta/\Delta_0$ itself must be $p$-divisible.
\end{claimproof}

We can now show: 

\begin{nclaim}\label{C:SEmixed}
$(K_1,v_1)$ is strongly dependent, Kaplansky and algebraically maximal. Moreover the value group $v_1K_1$ and residue field $K_1v_1$ are stably embedded as pure structures.
\end{nclaim}
\begin{claimproof}
The following argument is taken from \cite[Theorem 4.3.1]{johnson}. Since $(K,v)$ is sufficiently saturated, any countable chain of balls in $(K,v)$ has non-empty intersection. Therefore, the same is true for $(K_1,v_1)$. On the other hand, $\Delta/\Delta_0$ embeds into $\mathbb{R}$ and thus every cut has countable cofinality, consequently $(K_1,v_1)$ is spherically complete and thus algebraically maximal. It is obviously Kaplansky and hence, by Proposition \ref{P:transf-strong-dep}, $(K_1,v_1)$ is strongly dependent. The moreover is due, again, to \cite[Lemma 3.1]{JaSiTransfer}.
\end{claimproof}

It will be enough to show that the structure $(K,v_0,K_1,v_1, K_2,v_2)$ is strongly dependent, since $v$ is definable there. We apply Proposition \ref{P:transf-strong-dep} twice. Since $(K,v_0)$ is strongly dependent, and $K_1$ is stably embedded as a pure structure and $(K_1,v_1)$ is strongly dependent, $(K,v_0,K_1,v_1)$ is strongly dependent. Doing this again, we get our result. 
\end{proof}

\begin{corollary}\label{C:defectless}
Let $K$ be a strongly dependent field. Then for every henselian valuation $v$ on $K$, the valued field $(K,v)$ is defectless, and therefore algebraically maximal. 
\end{corollary}
\begin{proof}
By \cite[Theorem 4.3.2]{johnson} every strongly dependent $(K,v)$ is defectless. As defectless henselian fields are algebraically maximal \cite[Theorem 11.31]{FVK}, the corollary follows. 
\end{proof}

To finish the proof of Theorem \ref{T:mainfields} we need to show that in every strongly dependent henselian field the value group is stably embedded as an ordered abelian group and the residue field is stably embedded as pure field. In a different paper, we show that every strongly dependent henselian field admits elimination of field quantifiers, the result follows (see \cite{qe-kap}).

\begin{remark}
Theorem \ref{T:fields} can also be deduced from elimination of field quantifiers and \cite[Claim 1.17(2)]{strongdep}, see \cite{qe-kap}.
\end{remark}

We end with the following consequence of some of the results discussed in this paper. Note that it answers \cite[Question 9.9.3]{johnson} to the affirmative.
\begin{proposition}
Let $K$ be a dp-minimal field and $v$ a henselian valuation on $K$. Then $(K,v)$ is also dp-minimal.
\end{proposition}
\begin{proof}
We use Johnson's classification of dp-minimal valued fields \cite[Theorem 9.8.1]{johnson}. If $Kv$ is not algebraically closed then, by \cite[Theorem A]{JahNIP}, $v$ is definable in $K^{sh}$ and as $K^{sh}$ is dp-minimal, so is $(K,v)$. We may thus assume that $Kv$ is algebraically closed and hence dp-minimal. 
Since $|\mathbb{P}_{\infty}(vK)|\leq |\mathbb{P}_{\infty}(K^\times)|$ and $K^\times$ is a dp-minimal abelian group, $vK$ is also dp-minimal (see \cite[Proposition 5.1]{dpmin-val-fields}). By Corollary \ref{C:defectless}, $(K,v)$ is defectless. 

If $\mathrm{char}(K)=p$ then $Kv$ is p-divisible by Lemma \ref{L:Kaplansky}.

Finally, assume that $\mathrm{char}(K)=0$ and $\mathrm{char}(Kv)=p$ and let $\Delta,\Delta_0,K_1,K_2,v_0,v_1$ and $v_2$ be as in the proof of Theorem \ref{T:fields}. Since $Kv$ is infinite, so is $K_1v_1$ and by Claim 1 of the proof of Theorem \ref{T:main}, $\Delta/\Delta_0$ is $p$-divisible. Since $\Delta_0$ is p-divisible by Lemma \ref{L:Kaplansky}, $\Delta$, and hence $[0,v(p)]$, is $p$-divisible. Now we may apply \cite[Theorem 9.8.1]{johnson}.
\end{proof}

\bibliographystyle{plain}
\bibliography{Strongness}

\end{document}